\documentclass[reqno]{amsart}
\usepackage{graphicx}
\usepackage{amsthm,amsmath,verbatim,amssymb}
\usepackage{caption}
\usepackage{color}
\usepackage{arydshln}
\newtheorem{lem}{Lemma}
\usepackage{tikz}
\newtheorem{thm}{Theorem}
\newtheorem{cor}{Corollary}

\newcommand{\E}{\mathbb{E}}

\newcommand{\abs}[1]{\left\vert#1\right\vert}

\newcommand{\beq} {\begin{equation} }
	\newcommand{\eeq} {\end{equation} }

\begin{document}
	%\title[CLT of HPP]{Limit Theorems for Linear Statistics of Harmonic Point Process}
	
	\title[Smallest gaps]{Smallest gaps between zeros of stationary Gaussian processes}
	
	\author[Feng]{Renjie Feng}
	
	\author[G\"otze]{Friedrich G\"{o}tze}
	\author[Yao]{Dong Yao}

	\address{Sydney Mathematical Research Institute, The University of Sydeny, Australia.}
	\email{renjie.feng@sydney.edu.au}

	\address{Faculty of Mathematics, Bielefeld University, Germany.}
	
	\email{goetze@math.uni-bielefeld.de}
	
	\address{Research Institute of Mathematics, Jiangsu Normal University, China. }
	\email{dongyao@jsnu.edu.cn}

	\date{\today}
	\maketitle
	
	\begin{abstract}
		%	In this paper, we will study the smallest gaps of zeros of stationary Gaussian random processes.  We will prove that the smallest gaps after normalization will tend to a Poisson point process, and this implies the Gumbel distribution for the  smallest gaps in the limit. 
		
		%The results generalize these in \cite{FA, FXA} for random spherical harmonics   to random waves on any compact Riemannian manifold. 
		In this paper, we study the smallest gaps between successive zeros of nondegenerate smooth stationary centered Gaussian processes on the real line with the assumption that the covariance kernel $\kappa(x)$ and its derivatives  decay to 0 as $\abs{x}\to\infty$. We prove that, after rescaling, the smallest gaps converge to a Poisson point process with a specific rate. Moreover, the positions where these smallest gaps occur tend to a uniform distribution. Consequently, we can derive the limiting density for the $k$-th smallest gap. 
	\end{abstract}

	\section{Introduction}\label{intro}
	
	Let $f: \mathbb{R} \rightarrow \mathbb{R}$ be a  nondegenerate smooth stationary centered Gaussian process, i.e., for any fixed $t \in \mathbb{R}$, the shifted  process $x \mapsto f(x+t)$ has the same law as $f(\cdot)$.  The  correlation kernel of $f$ is  translation-invariant and is given by  a smooth function $\kappa: \mathbb{R} \rightarrow \mathbb{R}$ such that
	$$\mathbb{E}[f(x) f(y)]= \kappa(x-y) . $$
	%We say $f(x)$ is normalized if its correlation function $\kappa$ satisfies $\kappa(0)=1=-\kappa^{\prime \prime}(0)$.
	%Actually, any nondegenerate smooth stationary centered Gaussian process can be normalized. 
	For any  $n\in \mathbb{Z}_{\geq 0}$ and $\eta \in  \mathbb{R}_+$, we define  	$$\|\kappa \|_{n,\eta}:=\sup\Big\{ |\kappa^{(i)}(x)| \big | 0\leq i\leq n, |x|\geq\eta\Big\}.$$
	Throughout the paper, we always assume that the correlation kernel satisfies \beq\label{rapiddecay}\lim_{\eta\to\infty} \|\kappa \|_{n,\eta}\to 0, \quad \forall\, n \in \mathbb{Z}_{\geq 0}, \eeq
	i.e., the correlation kernel and its derivatives  decay to 0 as $\abs{x}\to\infty$. %A concrete example of such a correlation kernel is when $\kappa$ belongs to the Schwartz space.% The fact that $\kappa$ tends to 0 infinity ensures the non-degeneracy of the finite-dimensional marginal distributions of the process $f$. 
	%	Let $f:\mathbb R\to\mathbb R$ be a smooth stationary centered Gaussian process. Let $$\kappa: x\to \mathbb E [f(0)f(x)]$$ denote the 
	%correlation function of $f$. We assume that $f$ is normalized so that $$\kappa(0)=1=−\kappa''(0).$$  
	
	The zero set, denoted as $Z: = f^{-1}(0)$, is almost surely a closed discrete subset of $\mathbb{R}$. We define $Z_R$ as the set of  zeros that fall within the interval $[0, R]$, i.e., $$Z_R := [0, R] \cap Z.$$ Then the expected cardinality of $Z_R$ satisfies (e.g., Proposition 1.8 in \cite{AL})
	$$\mathbb{E}|Z_R| = R/\pi.$$
	Furthermore, Theorem 1.21 in \cite{AL} states that if $\kappa$ belongs to the Schwartz space (which satisfies \eqref{rapiddecay}), then for any test function $\phi(x)$ that is essentially bounded, integrable, and continuous almost everywhere, the linear statistics $\sum_{x_i \in Z_R} \phi(x_i)$ satisfies the central limit theorem as $R\to\infty$. As a result, the counting measure of the zeros exhibits the central limit theorem.
	
	In this paper, we  study the smallest gaps between successive zeros of smooth stationary centered Gaussian processes with the assumption \eqref{rapiddecay}.   Given zeros in $Z_R$, we can arrange them in ascending order as $x_1 < x_2 < \cdots < x_{|Z_R|}$. We define a two-dimensional point process of the rescaled gaps between successive zeros and their locations as follows: %$$\mathcal S_R:= \sum \delta_{R^{1/2}(x_{i+1}-x_i)}.$$then one has the convergence of the point process 
	\begin{equation}\label{sr}\mathcal S_R:= \sum_{x_i\in Z_R} \delta_{R^{1/2}(x_{i+1}-x_i),x_i/R }.
	\end{equation}
	%Given a bounded Borel set $A$, we define $\mathcal{S}_R(A)\subseteq Z_R$ as the set of zeros $x_i$ such that the difference between consecutive points, $x_{i+1} - x_i$, belongs to $R^{-1/2}A$, i.e.,  %let $\mathcal S_R(A)$ be the smallest gaps falling in $A$, i.e., 
	% \beq\label{sra}\mathcal S_R(A):=\{  x_i\in Z_R: x_{i+1}-x_i \in R^{-1/2}A\}.\eeq
	Our main result is  the convergence of the point process $\mathcal{S}_R$.
	Define a constant 
	\beq\label{constant}\alpha_0= \frac{\mathbb{E}\left[|f''(0)|^2 \mid f(0)=f'(0)=0\right]}{8\pi \sqrt{\operatorname{det}  \operatorname{Cov}\left(f(0),  f'(0)\right)}}. \eeq
	Clearly $\alpha_0$ only depends on the covariance kernel and its derivatives $\kappa^{(i)}(0)$  evaluated at $0$  up to the 4-th order.
	%For stationary Gaussian processes, the behavior of the 2-point correlation function has the asymptotic estimate (see \eqref{rho2eqsim} and Theorem \ref{vanishing} below),
	%\beq\label{sms}\rho_2(x,y)\simeq c |x-y|,\,\,\mbox{as}\,\, |x-y|\to 0,\eeq
	%where the constant $$c= \frac{\mathbb{E}\left[|f''(x)|^2 \mid f(x)=f'(x)=0\right]}{\sqrt{\operatorname{det}  \operatorname{Cov}\left(\left(f(x),  f'(x)\right)\right.}}$$ only depends on the correlation function and its derivatives at 0 up to order 4, by stationarity of the Gaussian process. 

	%A much deeper idea how the smallest gaps work in terms of the joint density, e.g., GOE case. 
	
	%The main result is, 

	\begin{thm}\label{main1} 
		%Let $\alpha_0$ be given by \eqref{constant}.
		As $R\to\infty$, the point process $\mathcal{S}_R$ converges in distribution to a Poisson point process $\mathcal{S}$ on $(0,\infty)\times [0,1]$ 
		with the rate 
		$$\mathbb E [\abs{\mathcal S(A\times B)}]=\left(\alpha_0\int_A udu\right) \times m(B) , $$
		for any bounded Borel sets $A\subset (0,\infty)$ and $B\subseteq [0,1]$. Here, $m(B)$ is the Lebesgue measure of $B$. 
	\end{thm}
	
	As a direct consequence of Theorem \ref{main1}, we obtain the following result. Let $\sigma_k$ be the $k$-th smallest gap, i.e., $\sigma_k$ is the $k$-th smallest number in the set of   gaps between successive zeros,
	$$\Big\{x_{i+1}-x_i, 1\leq i\leq \abs{Z_R}-1\Big\}.$$ 
	Let  $\pi_k$ be the position $x_i$ where $x_{i+1}-x_i$ is the $k$-th smallest gap.  
	\begin{cor}   
		For any $k\geq 1$ and $x\geq 0$, one has 
		$$\lim_{R\to\infty} \mathbb P\Big(\sqrt{\frac{\alpha_0}2}R^{1/2}\sigma_k\leq x
		\Big)=\int_0^x \frac{ y^{2k-1}}{(k-1)!}e^{-   y^2}dy.$$
		Furthermore, as $R\to\infty$,  $\pi_k/R$ is asymptotically independent of $\sigma_k$ and tends to the uniform distribution on $[0,1]$.
		
	\end{cor}
	
	%The proof of Theorem \ref{main1} is based on the method of modified point process developed by  Soshnikov  in \cite{So}, where Soshnikov studied the smallest gaps for determinantal point processes on the real line with translation-invariant kernels. %In the following, we outline the main steps to prove Theorem \ref{main1}.
	
	We outline the proof of Theorem \ref{main1} as follows. Fix any bounded Borel sets $A\subset (0,\infty)$ and $B\subseteq [0,1]$. %By the moment method, in order to establish the convergence towards the Poisson distribution, it is sufficient to prove the convergence of the factorial moment, \beq\label{factorial1}\lim_{R\to\infty}\mathbb E\Big(\frac{|\mathcal S_R(A\times B)|!}{(|\mathcal S_R(A\times B)|-k)!} \Big) =\Big(\alpha_0 m(B)\int_A udu\Big)^k\eeq for any $k\geq 1$.  
	%We will not directly prove this convergence for $\mathcal{S}_R$ itself; instead, we consider another auxiliary random point process. 
	Let $c_1$ be such that $A \subset(0, c_1)$, and set
	\begin{equation}\label{dr} d_{R}:=c_1 R^{-1 / 2}.\end{equation}
	Given  $Z_R$, for any $1\leq i \leq \abs{Z_R}$,  let $g_i$ be the indicator function of the event
	\begin{equation}\label{indic}
	\bigcap_{j:1\leq j\leq \abs{Z_R},j\neq i}\Big\{
\mbox{Either }x_{j+1}-x_j>d_R \mbox{ or } \abs{x_j-x_i}>R^{1/4}\Big\}.\end{equation}
Thus, $g_i=0$ if and only if there exists some $j\neq i$ such that both $x_{j+1}-x_j\leq d_R$ and $\abs{x_j-x_i}\leq R^{1/4}$ hold. 
	 We then define an auxiliary point process as
	\begin{equation}\label{tsra}\tilde {\mathcal S}_{R,A}:=\frac{1}{R} \Big\{  x_i\in Z_R: x_{i+1}-x_i \in R^{-1/2}A \mbox{ and }g_i=1\Big\}.\end{equation} 
	%In other words, $\tilde{\mathcal{S}}_{R,A}$ is a point process defined on $[0,1]$, consisting of rescaled zeros $x_i$ where there is exactly one zero $x_{i+1}$ falling within $x_i+ R^{-1/2}A$, and no other zero $x_j$ ($j\neq i$) satisfying $x_{j+1}-x_j\leq d_R$ falls within the range $[x_i-R^{1/4},x_i+R^{1/4}]$. 
	%Note that here we only count $x_i$ in $\tilde{S}_{R,A}$ but leaving out $x_{i+1}$. 
	%In fact,	for $R$ large enough so that $d_R<R^{1/4}$, if $x_i/R\in \tilde {\mathcal S}_{R,A}$, then necessarily $x_{i+1}/R\notin \tilde {\mathcal S}_{R,A}$. In other words,  $\tilde {\mathcal S}_{R,A}$ does not include consecutive points in $Z_R$.  
%	For a concrete example, consider the case $$R=16,\,
%	A=[1,6.8],\, c_1=7;$$
%	 and $$
%	Z_R=\{x_1,\ldots, x_6\}=\{1.6,\, 3.1,\, 5.2,\, 6.5,\, 8.1,\, 12\}.
%	$$
%	Then we  can compute that
%	\begin{equation}\label{g123}
%			g_1=g_5=g_6=1, \, g_2=g_3=g_4=0;
%	\end{equation}
%and 
%	\begin{equation}\label{x123}
%			x_2-x_1, x_4-x_3 \in R^{-1/2}A=[0.25, 1.7].
%	\end{equation}
%	Combining \eqref{g123} and \eqref{x123}, we get
%	$$
%	\tilde {\mathcal S}_{R,A}= \frac{1}{R}\{x_1\}=\frac{1}{16}\{1.6\}=\{0.1\}.
%	$$
	The definition of $\tilde {\mathcal S}_{R,A}$ implies that for any  $B\subseteq [0,1]$, we have
	$$  \Big{|} {\mathcal S}_R(A\times B)\Big{|}\geq  \abs{\tilde {\mathcal S}_{R,A}(B)}.$$ 
	%As we shall see, the modified point process is more convenient for analysis compared to the original one.  
	In Lemma \ref{nssn},  we will prove that the cardinalities of the point processes   $\tilde {\mathcal S}_{R,A}$ and  ${\mathcal S}_R(A  \times \cdot)$ are actually  approximately  the same in the sense that, 
	$$  \Big{|} {\mathcal S}_R(A\times [0,1])\Big{|}- \abs{\tilde {\mathcal S}_{R,A}([0,1])}\to 0$$
	in probability as $R\to\infty$. 
	Therefore, to prove Theorem \ref{main1}, by the moment method,  it suffices to prove that the factorial moments satisfy  %in Lemma \ref{bigdistance} that 
	\beq\label{factorial2}\lim_{R\to\infty}\mathbb E\Big(\frac{|\tilde{\mathcal S}_{R,A}( B)|!}{(|\tilde{\mathcal S}_{R,A}( B)|-k)!} \Big) =\Big(\alpha_0 m(B)\int_A udu\Big)^k, \eeq
	for any integer $k\geq1$. 	Let $\tilde \tau_k(x_1, ..., x_k)$ be the $k$-point correlation function of   $\tilde{\mathcal S}_{R,A}$, then by the definition of the correlation function,  \eqref{factorial2} is further equivalent to 
	\beq\label{tilderhok}\lim_{R\to\infty}\int_{B^k}\tilde \tau_k(x_1, ..., x_k)dx_1\cdots dx_k=\Big(\alpha_0 m(B)\int_A udu\Big)^k.\eeq
	To establish this, we first provide the upper bound \eqref{uppbd} and  the lower bound \eqref{lbtilde} for  $\tilde \tau_k(x_1, ..., x_k)$ in terms of the correlation functions of zeros of the stationary Gaussian process. Subsequently, in Lemma \ref{bigdistance}, we derive the   limit \eqref{spera} for  $\tilde \tau_k(x_1, ..., x_k)$ by proving that the upper and lower bounds converge asymptotically to the same value. This will lead to the limit \eqref{tilderhok}, thereby concluding the proof of Theorem \ref{main1}.
	
	In a companion paper \cite{FGY}, similar techniques are applied to study the smallest distances between complex zeros of Gaussian holomorphic sections of the positive holomorphic line bundle $L^n$ over compact Riemann surfaces (refer to \cite{BSZ1} for more details). The main results are that the smallest distances between zeros are of the order $n^{-3/4}$, and tend to a Poisson point process in the limit after rescaling. Moreover, the positions where the smallest distances occur tend to a uniform measure with respect to the volume form. These findings are universal and apply specifically to classical $\operatorname{SU(2)}$ random polynomials on the complex projective space $\mathbb{CP}^1\cong S^2$.
	
	It is worth mentioning that in \cite{So}, Soshnikov studied the smallest gaps of determinantal point processes on the real line with translation-invariant kernels. He proved that for points within the interval $[0, R]$, the smallest gaps decay with an order of $R^{-1/3}$ and also converge to a Poisson point process as $R\to\infty$. %The determinantal point processes exhibit the negative association property, i.e., particles are repulsive to each other. 
	
Recently, the largest gaps between consecutive zeros of stationary Gaussian  processes with rapidly decaying kernels were studied in \cite{FM}. It was proved that, after suitable rescaling, the largest gaps also converge to a Poisson process, which in turn determines their decay rate. Unlike the case of the smallest gaps, where one analyzes the splitting of the correlation function (e.g., \eqref{mainsplit}), establishing the Poisson limit for the largest gaps requires proving a splitting of the gap probabilities of the zeros, which is a technically demanding task.

	\bigskip
	
	\emph{Notation}. The symbol $C$ represents a positive constant; however, its specific value may vary from line to line. % We use the notation $a \simeq b$ to indicate that $a/b$ tends to 1 with respect to the corresponding limit. %Given a Borel set $A\subset \mathbb R$,  we will denote  $x+A:=\{x+a, a\in A\}$.

	\section{Correlation functions of zeros}\label{sec:2}

	In this section, we will provide a summary of the main results from \cite{AL} regarding the short-range and long-range behaviors of the correlation functions of zeros of a stationary Gaussian random process satisfying assumption \eqref{rapiddecay}.
	
	% vanishing order of the correlation functions for its zeros on the diagonal and the long range behavior when zeros are widely separated.

	%Let $k \in \mathbb{N}^*$, if $f$ is of class $\mathcal{C}^k$ then, as we stated above, the fact $\kappa(x) \underset{x \rightarrow+\infty}{\longrightarrow} 0$ is enough to ensure that $f$ satisfies the hypotheses of Theorem 1.13 at any point $y \in \mathbb{R}^k$. This condition is sufficient but not necessary, see Lemma A. 2 below.

	%We first review some
	%basic concepts about the factorial moments and the correlation functions of
	%a point process.  
	
	Given a simple point process on the real line $\mathbb{R}$ represented by   $$X=\sum_{i}\delta_{X_i},$$ one can define another point process on $\mathbb R^k$ as
	\begin{equation}\label{dis}
		X^{(k)}=\sum_{X_{i_1},\cdots, X_{i_k}\,\,\textrm{all distinct}}\delta_{(X_{i_1},\cdots, X_{i_k})}.\end{equation} 
	Then the $k$-point correlation function $\rho_k$ of $X$ is defined as follows, 
	\beq\label{corre}\mathbb E\abs{X^{(k)}(B_1\times \cdots \times B_k)} =\int_{B_1\times \cdots \times B_k}\rho_k(x_1,\cdots, x_k)dx_1\cdots dx_k, \eeq
	where   $B_1, \cdots, B_k$ are Borel sets in $\mathbb R$.  In particular,  taking $B_1=\cdots=B_k=B$, \eqref{corre} yields the $k$-th factorial moment of the number of points falling in $B$, 
	$$\mathbb E \left(\frac{\abs{X(B)}!}{(\abs{X(B)}-k)!}\right)=\int_{B^k}\rho_k(x_1,\cdots, x_k)dx_1\cdots dx_k.$$
	%where $B$ is a Borel set in $\mathbb R$. 
	
	Now, given a nondegenerate smooth stationary centered Gaussian process $f(x)$,  the $k$-point correlation function of zeros of this process is 
	given by the classical Kac-Rice formula  \cite{AT}.  We first define the diagonal set in $\mathbb{R}^k$ as follows,    $$\Delta_k=\Big\{(x_1,..., x_k): x_i=x_j \,\,
	\mbox{for some }\,\, i\neq j\Big\}.$$
	Then for $(x_1,.., x_k)\notin \Delta_k$, the Kac-Rice formula reads
	% Similarly, for $x_1\neq\cdots\neq x_k$, one has the following expression of the  $k$-point correlation function, 
	$$\rho_k(x_1,\cdots, x_k)= \frac1{(2\pi)^\frac k2} \frac{\mathbb E(\prod_{i=1}^k|f'(x_i)||f(x_1)=\cdots =f(x_k)=0)}{ \sqrt{ \det[\operatorname{Cov}(f(x_i), f(x_j))_{1\leq i, j\leq k}]}}.$$
	The denominator in the formula is nondegenerate for   $(x_1,..., x_k)\notin \Delta_k$  due to the assumption \eqref{rapiddecay}  of the  decay of the correlation kernel (e.g., Lemma 2.10 in \cite{AL}).

	%The above Kac-Rice formula is  for  the correlation functions defined on $\mathcal{\mathbb R}^k \backslash \Delta_k$. 
	
	It is important to note that difficulties arise on the diagonal $\Delta_k$, where both the numerator and the denominator tend to zero.  In \cite{AL}, based on the method of divided differences, the authors  determined the precise vanishing order of   $\rho_k$ near $\Delta_k$ assuming \eqref{rapiddecay}. To provide more details, consider $y=\left(y_i\right)_{1 \leq i\leq  k} \in \Delta_k$ and let $\mathcal{I}$ be the partition of  $\{1, 2,..., k\}$ defined by:
	$$
	\forall i, j \in\{1, \ldots, k\}, \quad 
	\exists I \in \mathcal{I},\,\,\,\{i, j\} \subset I
	\Longleftrightarrow 
	y_i=y_j .
	$$
	We denote by $y_I \in \mathbb{R}$ the common value of the $\left(y_i\right)_{i \in I}$ and define 
	\beq\label{lfunction}\begin{split}
		\ell(y):=&\left(\prod_{I \in \mathcal{I}} \prod_{i=0}^{|I|-1} \frac{i !}{|I| !}\right)\times \\ & \frac{\mathbb{E}\left[\prod_{I \in \mathcal{I}}\left|f^{(|I|)}\left(y_I\right)\right|^{|I|} \mid \forall I \in \mathcal{I}, \forall i \in\{0, \ldots,|I|-1\}, f^{(i)}\left(y_I\right)=0\right]}{(2 \pi)^{\frac{{k}}{2}} \operatorname{det}\left(\operatorname{Cov}\left(\left(f^{(i)}\left(y_I\right)\right)_{I \in \mathcal{I}, 0 \leqslant i<|I|}\right)\right)^{\frac{1}{2}}}. \end{split}
	\eeq 
	%note that the factor |I|! comes from the Taylor expan
	%where $\mathbb{E}\left[\prod_{I \in \mathcal{I}}\left|f^{(|I|)}\left(y_I\right)\right|^{|I|} \mid \forall I \in \mathcal{I}, \forall i \in\{0, \ldots,|I|-1\}, f^{(i)}\left(y_I\right)=0\right]$ stands for the conditional expectation of $\Pi_{I \in \mathcal{I}}\left|f^{(|I|)}\left(y_I\right)\right|^{|I|}$ given that $f^{(i)}\left(y_I\right)=0$ for all $I \in \mathcal{I}$ and $i \in\{0, \ldots,|I|-1\}$.
	Then, Theorem 1.13 in \cite{AL} establishes the convergence \beq\label{maincon}
	\lim_{x\to y}\left(\prod_{I \in \mathcal{I}} \prod_{\left\{(i, j) \in I^2 \mid i<j\right\}} \frac{1}{\left|x_i-x_j\right|}\right) \rho_k(x_1,.., x_k) =\ell(y).
	%\underset{x \rightarrow y}{\longrightarrow} \ell(y).
	\eeq
	In particular, consider the case $k=2$,  for any fixed $y\in \mathbb{R}$, we have 
	\beq\label{alp0}
	\lim_{x\to y} \frac{\rho_2(x,y)}{\abs{x-y}}=\ell(y)=
	\ell(0)= \frac{\mathbb{E}\left[|f''(0)|^2 \mid f(0)=f'(0)=0\right]}{8\pi \sqrt{\operatorname{det}  \operatorname{Cov}\left(f(0),  f'(0)\right)}}=\alpha_0,
	\eeq
	where we used the stationarity of the Gaussian process in the second equality, and the definition of $\alpha_0$ in \eqref{constant} in the last equality. 
	
	In Theorem 1.14 of \cite{AL}, the authors described  both the short-range and the long-range behaviors of the correlation functions.  For the short-range behavior,   there exists $C>0$ such that for all $x=\left(x_i\right)_{1 \leqslant i \leqslant k} \in \mathbb{R}^k \backslash \Delta_k$ one has 
	\beq\label{unbd}
	0 \leqslant \rho_k(x) \leqslant C \prod_{1 \leqslant i<j \leqslant k} \min \left(\left|x_i-x_j\right|, 1\right) .
	\eeq 
	This implies that the correlation functions vanish along  the diagonal, i.e., nearby zeros are repulsive to each other. For the long-range behavior, for all $\eta \geqslant 1$, for all partitions $\mathcal{I}$ of $\{1, 2,.., k\}$, for all $\left(x_i\right)_{1 \leqslant i \leqslant k} \in \mathbb{R}^k \backslash \Delta_k$ satisfying
	$$
	\forall I, J \in \mathcal{I} \text { such that } I \neq J,  \, \mbox{we have}\,\,  \forall i \in I,\,\, \forall j \in J_{}, \,\left|x_i-x_j\right|>\eta, \, \,
	$$
	one has the splitting property, 
	\beq\label{mainsplit}
	\rho_k(x_1,.., x_k)
	=	\prod_{I \in \mathcal{I}} \rho_{|I|}\left(\underline{x}_I\right)\left(1+O\left(\left(\|\kappa\|_{k, \eta}\right)^{\frac{1}{2}}\right)\right),
	\eeq 
	where the implicit constant involved in the error term $O\left(\left(\|\kappa\|_{k, \eta}\right)^{\frac{1}{2}}\right)$ depends only on $k$. Here, we denoted by $\underline{x}_I=\left(x_i\right)_{i \in I}$, for all $I \in \mathcal{I}$.	
	
	%For $x=\left(x_i\right)_{1 \leqslant i \leqslant k} \in \mathbb{R}^k \backslash \Delta_k$, although $x$ is not in the diagonal $\Delta_k$, but there are clusters of $x_i$'s that are closed to each other, this theorem implies that 
	%wherever the distance between two clusters is greater than $\eta$, then the two clusters are  asymptotic independent.  In particular, if $x_i-x_j\geq \eta$ for all $1\leq i<j\leq k$, then $\rho_1(x_1)\cdots \rho_1(x_k)=\rho_k(x_1,.., x_k)(1+O(\left(\|\kappa\|_{k, \eta}\right)^{\frac{1}{2}}))$. 
	
	%The appearance of $\|\kappa\|_{k, \eta}$ is quite natural when one tries to control the decay of the off diagonal elements of the covariance matrix of the Gaussian processes. For example, 
	%for $k=3$ and the case $(f(x), f(\simeq x), f(y))$, then the denominator is the covariance of matrix of the Gaussian vector $(f(x), f'(x), f(y))$, and numerator is the cexpectation of $|f''(x)|^2|f'(y)|$ where one has to estimate the covariance of $(f''(x), f'(y))$, which is to estimate the off diagonal element $\kappa^{(3)}(x-y)$. Therefore, if we expect the off diagonal element is decay to 0  fast enough, this is equivalent to assume  $\|\kappa\|_{3, \eta}\to 0$  as $\eta\to\infty$. 

	\section{Proof of Theorem \ref{main1}}
	
	% \subsection{No successive smallest gaps}
	%Actually, the above arguments can be applied to real Gaussian random polynomials. For example, let's take
	%$$p_n(x)=\sum_{j=0}^n {n\choose j}x^j,$$%
	%where $a_j$ are standard independent and identical distributed standard real Gaussian random variables. 

	%Given a bounded Borel subset $A\subset[0,\infty)$, 
	%for  the point process of zeros of the stationary Gaussian process falling in $[0, R]$, 
	%$$
	%\xi^{R}=\sum_{i=1}^{N(R)} \delta_{x_{i}},
	%$$ we define  its thinning $\tilde{\xi}^{R}$ obtained from $\xi^{R}$ by only keeping the points $x_{k}$ for which $\xi^{R}\left(x_{k}+A_{R}\right)=1$, where $A_R=R^{-1/2}A$. The following lemma means that $\mathcal S^{R}(A):=\langle \mathcal S^R, \chi_A\rangle$, the number of successive gaps such that $R^{1/2}(x_{i+1}-x_i)\in A$,  is properly estimated by the cardinality $| \tilde{\xi}^{R}|$. 
	In this section, we will prove the main result of Theorem \ref{main1}.
	Recall the definitions of $\mathcal{S}_{R}$ and  $\tilde{\mathcal{S}}_{R,A}$ in \S\ref{intro}. Let $\rho_k$ and $\tilde{\tau}_k$ denote the $k$-point correlation functions of   zero set $Z$ of the stationary Gaussian process and $\tilde{\mathcal{S}}_{R,A}$, respectively. We first have %The  following lemma demonstrates the asymptotic equivalence of the two point processes $\mathcal{S}_R(A\times \cdot)$ and $\tilde{\mathcal{S}}_{R,A}$.
	
	\begin{lem}\label{nssn}As $R\to\infty$, one has  the convergence,   
		$$  \Big{|} {\mathcal S}_R(A\times [0,1])\Big{|}- \abs{\tilde {\mathcal S}_{R,A}([0,1])}\to 0\,\,\, \mbox{ in probability}.$$
	\end{lem}
	
	\begin{proof} Recall that $c_1$ is chosen such that $A \subset(0, c_1)$, and we denote $d_{R}=c_1 R^{-1 / 2}$. By the definitions of $\mathcal{S}_R$ and $\tilde{S}_{R,A}$, we have  
		\begin{equation}\label{srdiff}
			0\leq |\mathcal S_{R}(A\times [0,1])|-|\tilde {\mathcal S}_{R,A}([0,1]) |   \leq |D_0|+|D_1|,
		\end{equation}
		where the sets $$
		D_0:=\Big\{  (x_1,x_2,x_3)\in Z_R^{(3)}: \abs{x_2-x_1}<d_R,  \abs{x_3-x_1}<d_R\Big \}
		$$
		and
		$$
		D_1:=\Big\{  (x_1,x_2,x_3,x_4)\in Z_R^{(4)}: \abs{x_2-x_1}<d_R,  \abs{x_4-x_3}<d_R, \abs{x_3-x_1}<R^{1/4} \Big \}.
		$$
		Here, $Z_R^{(3)}$ refers to the set of 3-tuples composed of pairwise distinct elements from the zeros set $Z_R$, as defined in  \eqref{dis}. The same applies to $Z_R^{(4)}$.  
		To prove that the integer-valued random variables $|D_0|$ and $|D_1|$ converge to 0 in law, we consider their expectations.  By the uniform upper bound for  the correlation functions in \eqref{unbd} together with the stationarity of the Gaussian process,  we have 
		\beq \label{d0}\begin{split}
			\mathbb E|D_0|=&\int_{0}^{R} \mathrm{d}x_1  \int_{\left(x_1-d_R, x_1+d_{R}\right)\times (x_1-d_R,x_1+d_R ) } \rho_{3} \left(x_1, x_{2}, x_{3}\right) \mathrm{d} x_{2} \mathrm{~d} x_{3}\\ \leq &C \int_{0}^{R}   dx_1\int_{(-d_R, d_R)\times (-d_R, d_R) }  |x_2||x_3||x_2-x_3|dx_2dx_3. 
		\end{split}\eeq
		%By stationarity the function $\ell(x_0)$ can be expressed by the covariance function and its derivatives evaluated at 0, and thus it's uniformly bounded. 
		This integral can be bounded from above by $O(Rd_R^5 )=O(R^{-3/2})$, which tends to 0 as $R\to\infty$. Similarly, we   have 
		\beq \label{d1}\begin{split}
			&\mathbb E|D_1|\\=&\int_{0}^{R} \mathrm{d}x_1 
			\int_{x_1-R^{1/4}}^{x_1+R^{1/4}} \mathrm{d}x_3 
			\int_{\left(x_1-d_R, x_1+d_{R}\right)\times (x_3-d_R,x_3+d_R ) } \rho_{4} \left(x_1, x_{2}, x_{3},x_4\right) \mathrm{d} x_{2} \mathrm{~d} x_{4}\\ \leq &C \int_{0}^{R}   dx_1   	\int_{x_1-R^{1/4}}^{x_1+R^{1/4}} \mathrm{d}x_3 
			\int_{(-d_R, d_R)\times (-d_R, d_R) }  |x_2||x_4|dx_2dx_4\\
			\leq & C R^{5/4} d_R^4=O(R^{-3/4}),
		\end{split}\eeq which also tends to 0 as $R\to\infty$. 
		This  completes the proof of Lemma \ref{nssn} by \eqref{srdiff}.  \end{proof}

	As explained in \S\ref{intro}, the remaining effort is to prove  the following convergence of the factorial moment of $\tilde{\mathcal{S}}_{R,A}(B)$, 
	$$ \lim_{R\to\infty}\int_{ B^k}\tilde \tau_k(x_1, ..., x_k)dx_1...dx_k=\Big(\alpha_0 m(B)\int_A udu\Big)^k,$$
	for any integer $k\geq1$. 
	
	To this end, we first define two  subsets of $\Omega^k$, on which we can obtain some upper and lower bounds for $\tilde{\tau}_k$.  Define $$ \Omega_k=\frac{1}{R}\Big\{(x_1,..,x_k)\in [0,R]^k: \,\min_{i\neq j}|x_i-x_j| \geq R^{1/4}\Big \}, $$ 
	and let $$(\Omega_k)^c=[0,1]^k\backslash \Omega_k.$$ 
	By the definition of $\tilde{\mathcal{S}}_{R,A}$, we have \beq\label{idds}\tilde{\tau}_k \equiv 0\,\, \,\mbox{on} \,\,\,(\Omega_k)^c.\eeq
	We now define another subset of ${\Omega}^k$ as, 
	$$ \tilde{\Omega}_k=\frac{1}{R}\Big\{(x_1,..,x_k)\in [0,R-d_R]^k: \,\min_{i\neq j}|x_i-x_j| \geq R^{1/4}+d_R \Big\}.$$
	Given any $(x_1,\ldots, x_k)\in \tilde{\Omega}_k$,  we set
	$$
	D_2(x_1,\ldots, x_k )= \bigcup_{i=1}^k  [Rx_i-d_R, Rx_i+d_R],
	$$
	and
	$$
	D_3(x_1,\ldots, x_k )= \bigcup_{i=1}^k  \Big\{(z_1,z_2):0\leq z_2-z_1\leq d_R, \abs{z_1-Rx_i}<R^{1/4}\Big\}.
	$$

	\begin{lem}\label{lem:upplow0}
	We have the following upper bound on $\Omega_k$,
	\begin{equation}\label{uppbd}
		\begin{split}&\tilde{\tau}_k(x_1,\ldots, x_k)  \\   \leq  &  R^{k}
			\int_{Rx_1+A_R} \cdots \int_{Rx_k+A_R} \rho_{2k}(Rx_1,y_1,\ldots, Rx_k,y_k)dy_1\cdots dy_k :=L(R),
		\end{split}
	\end{equation}where 
\begin{equation}\label{defar}
		A_R:=R^{-1/2}A= \big\{R^{-1/2}x:x\in A\big\}.
\end{equation}
	We also  have the following lower bound on $\tilde{\Omega}_k$, 
	\begin{equation}\label{lbtilde}
		\begin{split}
			&	\quad\,\,\tilde{\tau}_k(x_1,\ldots, x_k)  \\&
			\geq  R^{k} 
			\int_{Rx_1+A_R} \cdots \int_{Rx_k+A_R} \rho_{2k}(Rx_1,y_1,\ldots, Rx_k,y_k)dy_1\cdots dy_k\\
			&-R^{k} 	\int_{Rx_1+A_R} \cdots \int_{Rx_k+A_R} \int_{D_2} \rho_{2k+1}(Rx_1,y_1,\ldots, Rx_k,y_k,z)dy_1\cdots dy_k dz\\
			&-R^{k} 	\int_{Rx_1+A_R} \cdots \int_{Rx_k+A_R} \int_{D_3} \rho_{2k+2}(Rx_1,y_1,\ldots, Rx_k,y_k,z_1,z_2)dy_1\cdots dy_k dz_1dz_2
			\\ &:= L(R)-E_2(R)-E_3(R).
		\end{split}
	\end{equation}

\end{lem} 

\begin{proof}[Proof of Lemma \ref{lem:upplow0}]
We consider the scaled point process $$
\mathcal S_{R,A}:=R\tilde{\mathcal S}_{R,A}=\Big\{  x_i\in Z_R: x_{i+1}-x_i \in R^{-1/2}A \mbox{ and }g_i=1\Big\},
$$
and denote by $\tau_k$ its $k$-point correlation function. Then we have the relation
\begin{equation}\label{scale}
	\tilde{\tau}_k(x_1,\ldots, x_k)=
	R^k \tau_k(Rx_1,\ldots,Rx_k).
\end{equation}
Given any $\delta>0$ and $w_1,\ldots, w_k \in R\Omega_k$, we define a set $D(w_1,\ldots, w_k, \delta)$ to be
\begin{equation*}
	\begin{split}
\big\{
(x_1,y_1,\ldots,x_k,y_k)\in Z_R^{(2k)}:\forall \leq  i\leq k,\,
x_i\in [w_i,w_i+\delta],
y_i\in x_i+A_R
\big\}.
\end{split} 
\end{equation*}
For $\delta$ small, we have
\begin{equation*}
	\Pi_{i=1}^k\mathcal{S}_{R,A}\left([w_i,w_i+\delta]\right) \leq \abs{D(w_1,\ldots, w_k, \delta)}.
\end{equation*}
Consequently, 
\begin{equation*}
	\begin{split}
	\tau_k(w_1,\ldots, w_k)&\leq 
\limsup_{\delta \to 0}	\frac{1}{\delta^k}
\E\abs{D(w_1,\ldots, w_k, \delta)}\\
&=	\int_{w_1+A_R} \cdots \int_{w_k+A_R} \rho_{2k}(w_1,y_1,\ldots, w_k,y_k)dy_1\cdots dy_k.
\end{split}
\end{equation*}
Equation \eqref{uppbd} now follows from 
this and the scaling relation \eqref{scale}.
The lower bound \eqref{lbtilde} can be proved similarly by observing that
\begin{equation*}
	\begin{split}
\Pi_{i=1}^k\mathcal{S}_{R,A}\left([w_i,w_i+\delta]\right) \geq &\abs{D(w_1,\ldots, w_k, \delta)}\\
&-\abs{D'(w_1,\ldots, w_k, \delta)}
- \abs{D''(w_1,\ldots, w_k, \delta)},
\end{split} 
\end{equation*}
where the two sets $D'(w_1,\ldots, w_k, \delta)$ and $D''(w_1,\ldots, w_k, \delta)$ are given by
\begin{equation*}
	\begin{split}
		\Big\{
		(x_1,y_1,\ldots,x_k,y_k,z)\in Z_R^{(2k+1)}:& \forall\, 1\leq i\leq k, 
		x_i\in [w_i,w_i+\delta],
		y_i\in x_i+A_R ;\\
		& z\in D_3\left(w_1/R, \ldots, w_k/R\right)
		\Big\},
	\end{split}
\end{equation*}
and
\begin{equation*}
	\begin{split}
		\Big\{(x_1,y_1,\ldots,x_k,y_k,z_1,z_2)\in Z_R^{(2k+2)}: &\forall 1\leq  i\leq k, 
		x_i\in [w_i,w_i+\delta],\\
		y_i\in x_i+A_R;  \,
		& (z_1,z_2)\in D_4\left(w_1/R, \ldots, w_k/R\right) \Big\},
	\end{split}
\end{equation*}
 respectively. This implies \eqref{lbtilde}.
\end{proof}

	By \eqref{uppbd} and \eqref{lbtilde}, 
	we can prove the following lemma. 
	\begin{lem}\label{bigdistance}
		There exists a constant  $C>0$ such that   
		\begin{equation}\label{unbd2}
			\tilde{\tau}_k(x_1,\ldots, x_k)\leq C\,\, \mbox{ on }\,\,\Omega_k.
		\end{equation}
		Moreover, as $R\to\infty$, one has the convergence,
		\beq\label{spera}
		\tilde{\tau}_{k}^{}\left(x_{1}, \ldots, x_{k}\right) \to \left(\alpha_0 \int_{A} u{d} u\right)^{k}\,\, \mbox{ uniformly on }\,\,\tilde{\Omega}_k.  \eeq 
	\end{lem} 
	\begin{proof}
		We first prove \eqref{unbd2}. 
		Recall the definitions of $d_R$  and $A_R$ in \eqref{dr} and \eqref{defar}. We have $0\leq x\leq d_R$ for $x\in A_R$. Therefore,
		by \eqref{unbd} and \eqref{uppbd}, we have
		\begin{align*}
			\tilde{\tau}_k(x_1,\ldots, x_k)&\leq C R^k \int_{Rx_1+A_R}\abs{Rx_1-y_1}dy_1 \cdots \int_{Rx_k+A_R}\abs{Rx_k-y_k}dy_k\\
			&\leq C R^k \int_{A_R} d_R dy_1 \cdots \int_{A_R} d_R dy_1\\
			&\leq C R^k (R^{-1})^k =C. 		 
		\end{align*}
		%corresponding to $m=0$ in \eqref{inex} gives the expected asymptotic estiamte.
		We now turn to proving \eqref{spera}.  By the upper bound \eqref{uppbd}  and the lower bound \eqref{lbtilde}, it is sufficient to prove the following uniform 
		convergences on $\tilde{\Omega}_k$,   $$\lim_{R\to\infty}L(R)= \left(\alpha_0 \int_{A} u{d} u\right)^{k}  \,\,\mbox{and}\,\,\lim_{R\to\infty} E_2(R)=\lim_{R\to\infty}E_3(R)=0.$$ 
		To find the limit of $L(R)$, given $(x_1, \dots, x_k) \in \tilde{\Omega}_k$ and $y_i\in x_i+A_R$ ($1\leq i\leq k$), we can partition the $2k$ points $Rx_1,\ldots, Rx_k, y_1,\ldots, y_k$ into 
		$k$ clusters $\mathcal{I} = \{\{Rx_1, y_1\}, \dots, \{Rx_k, y_k\}\}$. 
		The distance between any two clusters of $\mathcal{I}$ is at least $R^{1/4}$. 
		Thus, using \eqref{mainsplit}, we obtain the following uniform estimate:
		$$
		\rho_{2k}(Rx_1, .., Rx_k, y_1,.., y_k)=\rho_2(Rx_1, y_1)\cdots \rho_2(Rx_k, y_k)\left(1+O\left(\left(\|\kappa\|_{2k, R^{1/4}}\right)^{\frac{1}{2}}\right)\right).$$
		%Here we only need the derivatives of $\kappa$ up to oder 4, as illustrated by the divided differences,  in our case of cluster of pairs, we only need to consider the Kac-Rice formula involved Gaussian random vectors $(f(x_i), f'(x_i), f''(x_i))$ for $i=1,.., k$. 
		Now we study the integration of the 2-point correlation function.  By the stationarity of the Gaussian process, for any $x\in \mathbb{R}$, we have 
		\beq \label{dsdsdsds}\int_{ x+R^{-1/2}A} \rho_2(x,y)dy =  \int_{ R^{-1/2}A}\rho_2(0,y)dy= R^{-1/2}\int_{ A}\rho_2(0,R^{-1/2}y)dy.\eeq
		Applying  the  limit \eqref{maincon} for the 2-point correlation function, we obtain  $ \rho_2(0, R^{-1/2}y)\simeq \alpha_0 R^{-1/2}y$ uniformly in $y\in A$ for sufficiently large $R$, and $\alpha_0$ is equal to $\ell(0)$  as given in \eqref{alp0}.
		Hence, by \eqref{dsdsdsds}, we further have 
		\beq\label{caserho2}\begin{split}\lim_{R\to\infty}R
			\int_{y\in x+R^{-1/2}A} \rho_2(x,y)dy=&\lim_{R\to\infty}R^{1/2}\int_{ A}\rho_2(0,R^{-1/2}y)dy\\=&\alpha_0  \int_{y\in A}ydy.\end{split} \eeq
		Therefore, we have the following uniform convergence
		for  $(x_1,\ldots, x_k)\in \tilde{\Omega}_k$, 
		\begin{equation}\label{mainterm}
			\begin{split}
				& \lim_{R\to\infty}R^k \int_{(Rx_{1}+A_{R})\times  \cdots  \times ({Rx_{k}+A_{R})}} \rho_{2k}(Rx_1, .., Rx_k, y_1,.., y_k){d} y_{1}  \cdots {d} y_{k}\\=&\lim_{R\to\infty} \Big[\prod_{i=1}^kR \int_{Rx_{i}+A_{R}}\rho_2(Rx_i, y_i)dy_i \Big]\left(1+O\left(\left(\|\kappa\|_{2k, R^{1/4}}\right)^{\frac{1}{2}}\right)\right)\\=& \left(\alpha_0 \int_{A} u^{} d u\right)^{k},
			\end{split}
		\end{equation}
		where in the last step, one has $ \|\kappa\|_{2k, R^{1/4}}\to 0$ as $R\to\infty$ by the assumption \eqref{rapiddecay}.   
		This gives the limit of $L(R)$, as desired. 
		
		% For simplicity, we denote by $E_2$ and $E_3$ the third line and the fourth line (neglecting the minus sign) in \eqref{lbtilde}, respectively. 
		%Now we denote the last two integrations in  \eqref{lbtilde} as, 
		%$$E_2=R^{k} 	\int_{Rx_1+A_R} \cdots \int_{Rx_k+A_R} \int_{D_2} \rho_{2k+1}(Rx_1,y_1,\ldots, Rx_k,y_k,z)dy_1\cdots dy_k dz$$
		%and $$E_3=R^{k} 	\int_{Rx_1+A_R} \cdots \int_{Rx_k+A_R} \int_{D_3} \rho_{2k+2}(Rx_1,y_1,\ldots, Rx_k,y_k,z_1,z_2)dy_1\cdots dy_k dz_1dz_2.$$
		%In the followings,  we will  prove that $E_2$ and $E_3$ tend to 0 as $R\to\infty$, which are asymptotically negligible compared to the limit of $L$.  And thus we complete the proof of \eqref{spera}. 
		
		Now we consider the error term $E_2(R)$. Without loss of generality, we may assume $z\in [Rx_k-d_R, Rx_k+d_R]$.
		And thus we can divide the $2k+1$ points $Rx_1,y_1,\ldots, Rx_k,y_k,z$ into $k$ clusters
		$
		\mathcal{I}=\{\{Rx_1,y_1\},\ldots,\{Rx_k,y_k,z\}  \}.
		$
		For any two clusters in $\mathcal{I}$, their distance is at least $R^{1/4}-d_R\geq R^{1/4}/2$.  Consequently, by the splitting property \eqref{mainsplit} again, one has
		\begin{equation}\label{2k+1}
			\begin{split}
				&\rho_{2k+1}(Rx_1,y_1,\ldots, Rx_k,y_k,z)\\
				=&\rho_3(Rx_k,y_k,z)\prod_{i=1}^{k-1}\rho_2(Rx_i,y_i) \left(1+O\left(\left(\|\kappa\|_{2k+1, R^{1/4}/2}\right)^{\frac{1}{2}}\right)\right)\\
				\leq  & C \rho_3(Rx_k,y_k,z)\prod_{i=1}^{k-1}\rho_2(Rx_i,y_i).
			\end{split}
		\end{equation}
		The estimate for $\E|D_0|$ in \eqref{d0} implies that the integration of $R\rho_3(Rx_k,y_k,z)$ (with respect to $y_k$ and $z$) tends to 0 uniformly in $x_k\in [0,1]$, while \eqref{caserho2} shows that the integration of $R\rho_2(Rx_i,y_i)$ (with respect to $y_i$) gives rise to a uniformly bounded factor.
		Hence, $E_2(R)$ will converges to 0 as $R\to\infty$, uniformly in $x_1,\ldots, x_k$.

		To estimate $E_3(R)$,  we further decompose $D_3(x_1,\ldots, x_k)$ into two disjoint sets, 
		$$
		D_4(x_1,\ldots, x_k):=\bigcup_{i=1}^k 
		\Big\{(z_1,z_2):  z_1\in [Rx_i-R^{1/4}/2, Rx_i+R^{1/4}/2],
		0\leq z_2-z_1\leq d_R\Big\},
		$$ and
		$$
		D_5(x_1,\ldots, x_k):=D_3(x_1,\ldots, x_k)\backslash D_4(x_1,\ldots, x_k).
		$$
		%We now estimate the integration of $\rho_{2k+2}(Rx_1,y_1,\ldots, Rx_k,y_k,z_1,z_2)$ in $E_3$ according to $(z_1,z_2)\in D_4$ or $(z_1,z_2)\in D_5$. 
		And thus we can rewrite 
		\begin{align*}E_3(R)=&R^{k} 	\int_{Rx_1+A_R} \cdots \int_{Rx_k+A_R} \int_{D_4} \rho_{2k+2}(Rx_1,y_1,\ldots, Rx_k,y_k,z_1,z_2) \cdots \\ &+R^{k} 	\int_{Rx_1+A_R} \cdots \int_{Rx_k+A_R} \int_{D_5} \rho_{2k+2}(Rx_1,y_1,\ldots, Rx_k,y_k,z_1,z_2) \cdots \\ :=&E_3'(R)+E_3''(R)
		\end{align*}
		For the error term $E_3'(R)$, without loss of generality,  we may assume $z_1\in [Rx_k-R^{1/4}/2, Rx_k+R^{1/4}/2]$. Then we have a partition
		$
		\mathcal{I}=$ $\{\{Rx_1,y_1\},\ldots,\{Rx_k,y_k,z_1,z_2\} \},
		$
		and for any two clusters in $\mathcal{I}$, their distance is at least $R^{1/4}/2-2d_R\geq R^{1/4}/3$. Similarly to \eqref{2k+1}, we have  
		$$
		\rho_{2k+2}(Rx_1,y_1,\ldots, Rx_k,y_k,z_1, z_2)
		\leq   C \rho_4(Rx_k,y_k,z_1,z_2)\prod_{i=1}^{k-1}\rho_2(Rx_i,y_i).
		$$
		The analysis for $\E|D_1|$ in \eqref{d1} implies  
		$$
		R\int_{Rx_k-d_R}^{Rx_k+d_R}dy_k \int_{Rx_k-R^{1/4}/2}^{Rx_k+R^{1/4}/2}dz_1 \int_{z_1-d_R}^{z_1+d_R} \rho_4(Rx_k,y_k,z_1,z_2) dz_2\leq C R^{-3/4},
		$$
		which goes to 0 as $R\to\infty$. 
		As before,  the integration of $R\rho_2(Rx_i,y_i)$ is bounded for all $i=1,.., k-1$, i.e., 
		$$
		R^{k-1}\int_{Rx_1-d_R}^{Rx_1+d_R}	\cdots  \int_{Rx_k-d_R}^{Rx_k+d_R}	\prod_{i=1}^{k-1}\rho_2(Rx_i,y_i) dy_1\cdots dy_{k-1}\leq C.
		$$
		Combining the above two upper bounds, we conclude that $E_3'(R)$ tends to 0.
		
		For the error term $E_3''(R)$, we have a partition  
		$
		\mathcal{I}=\{\{Rx_1,y_1\},\ldots,\{Rx_k,y_k\},\{z_1,z_2\}\}.
		$ In addition, the distance between any two clusters is at least $R^{1/4}/2-2d_R\geq R^{1/4}/3$. Therefore, we get
		$$
		\rho_{2k+2}(Rx_1,y_1,\ldots, Rx_k,y_k,z_1, z_2)
		\leq   C\rho_2(z_1,z_2) \prod_{i=1}^{k}\rho_2(Rx_i,y_i) .
		$$
		Again, we have
		$$
		R^k\int_{Rx_1-d_R}^{Rx_1+d_R}	\cdots  \int_{Rx_k-d_R}^{Rx_k+d_R}	\prod_{i=1}^{k}\rho_2(Rx_i,y_i) dy_1\cdots dy_k\leq C.
		$$
		For the integration of $\rho_2(z_1,z_2)$, we have the upper bound
		\begin{align*} 
			\int_{D_5} \rho_2(z_1,z_2)dz_1dz_2 
			&\leq \sum_{i=1}^k \int_{Rx_i-R^{1/4}}^{Rx_i+R^{1/4}}dz_1 \int_{z_1-d_R}^{z_1+d_R}
			\rho_2(z_1,z_2)dz_2\\
			&\leq C R^{1/4} R^{-1}=O(R^{-3/4}).
		\end{align*}
		The above two upper bounds will imply that $E_3''(R)$ also tends to 0. Therefore, as $R\to\infty$, 
		we have
		$E_3(R)$ tends to 0 uniformly in $(x_1,\ldots, x_k)\in \tilde{\Omega}_k$. This completes the proof of Lemma \ref{bigdistance}. 
	\end{proof}  
	Lemma \ref{bigdistance} will imply \eqref{tilderhok} directly. Indeed, we now show that 
	\beq\label{finald}\begin{split}&\lim_{R\to\infty}\int_{B^k} \tilde\tau_k(x_1,.., x_k)dx_1\cdots dx_k\\ =&\lim_{R\to\infty}\int_{B^k \cap  \tilde{\Omega}_k} +\lim_{R\to\infty}\int_{B^k\cap( \Omega_k \backslash \tilde{\Omega}_k) }+\lim_{R\to\infty}\int_{B^k\cap( \Omega_k )^c} \\=& \left(\alpha_0 m(B)\int_A udu\right)^k.\end{split}\eeq
	To prove this,  by the limit \eqref{spera} of $ \tilde\tau_k(x_1,.., x_k)$ on  $\tilde{\Omega}_k$ 
	in Lemma \ref{bigdistance}, we first have $$\lim_{R\to\infty}\int_{B^k \cap  \tilde{\Omega}_k}  \tilde\tau_k(x_1,.., x_k)dx_1\cdots dx_k=\left(\alpha_0 m(B)\int_A udu\right)^k,$$
	where we have used the fact that  $\tilde{\Omega}_k$ tends to the whole space $[0,1]^k$ as $R\to\infty$. 
	Recall \eqref{unbd2} in Lemma \ref{bigdistance}, where $\tilde{\tau}_k$ is uniformly bounded on ${\Omega}_k$. This, combined with the fact that the Lebesgue measure of the set $\Omega_k \backslash \tilde{\Omega}_k$ converges to 0 as $R\to\infty$, implies that
	$$\lim_{R\to\infty}\int_{B^k\cap( \Omega_k \backslash \tilde{\Omega}_k) } \tilde\tau_k(x_1,.., x_k)dx_1\cdots dx_k=0.$$
	Recall \eqref{idds},  where  $\tilde\tau_k\equiv0$ on $( \Omega_k )^c$. This implies  $$\lim_{R\to\infty}\int_{B^k\cap( \Omega_k )^c} \tilde\tau_k(x_1,.., x_k)dx_1\cdots dx_k=0.$$
	Therefore, we have established \eqref{finald}, which completes the proof of Theorem \ref{main1}.
	% $\tilde\tau_k$ is uniformly bounded on  $\tilde{\Omega}_k$,  the Lebesgue measures of the sets $(\tilde{\Omega}_k)^c$ and $\Omega_k \backslash\tilde{\Omega}_k$ both converge to 0 as $R\to\infty$ and $\tilde\tau_k\equiv 0$ on $(\Omega_k )^c$.  
	
	\section*{Acknowledgements}
	We thank the anonymous referees for helpful comments. Dong Yao is supported by
	National Key R$\&$D program of China (No. 2023YFA1010101),
	 NSFC grant (No. 122011256),
	NSF of Jiangsu Province grant (No. BK20220677),
	and Jiangsu Normal University Start-up grant 
	(No. 21XFRX018).

\end{document}